\documentclass[reqno]{amsart}

\theoremstyle{definition}

\usepackage[T1]{fontenc}
\usepackage[utf8]{inputenc}
\usepackage{lmodern}
\usepackage{textcomp}
\usepackage{microtype}
\usepackage{tikz}
\usepackage{mathtools, bm, nccmath}

\usepackage[english]{babel}
\usepackage{csquotes}
\usepackage[export]{adjustbox}
\usepackage{circuitikz}

\usepackage{tabularx}

\usepackage{graphicx}
\usepackage{setspace}

\usepackage{amsmath}
\usepackage{amsfonts}
\usepackage{amssymb}
\usepackage{amsthm}
\usepackage{xcolor}%
\usepackage{soul}
\usepackage{graphicx} 
\usepackage{lipsum}
\numberwithin{equation}{section}

\newcommand*{\N}{\mathbb{N}}

\newcommand*{\R}{\mathbb{R}}
\newcommand*{\C}{\mathbb{C}}

\newcommand*{\E}{\mathbb{E}}

\newcommand{\probP}{\text{I\kern-0.15em P}}

\usepackage{thmtools}

\declaretheorem[
	name=Theorem,
	numberwithin=section
	]{thm}
 
\declaretheorem[
	name=Lemma,
	sibling=thm,
	]{lem}

\declaretheorem[
	name=Remark,
	style=remark,
	numbered=no
	]{rem}

\usepackage[pdftex,pdfpagelabels]{hyperref}

\numberwithin{equation}{section}

\makeatletter
\@namedef{subjclassname@2020}{\textup{2020} Mathematics Subject Classification}
\makeatother

\calclayout

\begin{document}
\title[Kadec Pe{\l}czy\'nski theorem in Orlicz spaces]
{A Marcinkiewicz-Zygmund inequality and the Kadec Pe{\l}czyn\'ski theorem in Orlicz spaces}

\author[I. Berkes, E. Stefanescu, and R. Tichy ]{Istvan Berkes, Eduard Stefanescu, and Robert Tichy}
\address{Institut f\"ur Analysis und Zahlentheorie, TU Graz, Steyrergasse 30, 8010 Graz, Austria}
\email{\href{mailto:berkes@tugraz.at }{berkes@tugraz.at }}
\email{\href{mailto:eduard.stefanescu@tugraz.at}{eduard.stefanescu@tugraz.at}}
\email{\href{mailto:tichy@tugraz.at}{tichy@tugraz.at}}


\begin{abstract}
In this paper, we extend the Marcinkiewicz--Zygmund inequality to the setting of Orlicz and Lorentz spaces. Furthermore, we generalize a Kadec--Pe{\l}czy\'nski-type result- originally established by the first and third authors for $L^p$ spaces with $1 \le p < 2$ - to a broader class of Orlicz spaces defined via Young functions $\psi$ satisfying $x \le \psi(x) \le x^2$.
\end{abstract}

\maketitle

\section{Introduction}

Inequalities for the $L^p$ norm of partial sums of independent random variables, such as Khinchin's inequality \cite{KH}, the Marczinkiewicz-Zygmund inequality \cite{MZ2,MZ3}, and Rosenthal's inequality \cite{R1,R2}  play an important role in Banach space theory, in particular in the study of the subspace structure of $L^p$ spaces. For basic structure theorems see e.g.\ Kadec and Pe{\l}czy\'nski  \cite{KP}, Rosenthal \cite{R1,R2}, Gaposhkin \cite{GA} and the references there. The purpose of this paper is to extend the Khinchin inequality \cite{KH} and the Marcinkiewicz--Zygmund inequality \cite{MZ1} for independent random variables to Orlicz and Lorentz spaces, thus broadening the classical \(L^p\)-structure theory to these more general settings. Such results were established in full generality by Astashkin; see \cite{A1,A2}. Our contribution here is to provide a simple and direct proof for the particular cases of Orlicz and Lorentz spaces.
Furthermore, in Section 3 we also prove a Kadec-Pe{\l}czy\'nski type result for Orlicz spaces.
Results on Khinchin's inequality in Orlicz spaces generated by the Young function $\psi_2:=e^{-x^2}-1$ can be found in \cite{PE,PE2}. Other related results are \cite{ALM,CI,PAW}.


To make this precise, we recall the notion of equivalence of sequences in Banach spaces. Let \( X \) and \( Y \) be Banach spaces. Two sequences \( (x_n) \subset X \) and \( (y_n) \subset Y \) are said to be equivalent if there exists a constant \( K \geq 1 \) such that for all finitely supported scalar sequences \( (a_n) \), we have
\[
K^{-1} \left\| \sum a_n x_n \right\|_X \leq \left\| \sum a_n y_n \right\|_Y \leq K \left\| \sum a_n x_n \right\|_X.
\]
This notion captures the idea that the sequences induce isomorphic linear structures within their respective spaces. Equivalence of sequences plays a key role in understanding the geometry and basis structure of Banach spaces.


\section{Basic properties}
We use the Vinogradov symbol \(\ll\) to denote an inequality up to a constant, i.e., \(A \ll B\) means that \(A \leq C B\) for some constant \(C > 0\). If a parameter appears in the subscript, such as \(\ll_\alpha\), this indicates that the implicit constant depends on \(\alpha\); $\simeq$ means equal up to constants.
The notation $g(\cdot)$ denotes, the function $x\mapsto g(x)$, e.g. $(\cdot)$ denotes the map $x\mapsto x$. 

A Young-function or Orlicz-function is a convex map $\psi:[0,\infty)\to[0,\infty)$ with $\psi(x)/x\to\infty$ for $x\to\infty$ and $\psi(x)/x\to 0$ for $x\to 0$. We denote its Young complement by $\psi^*$, which is a Young function with the property $\psi^*(x)=\int_0^{|x|}\sup\{s:\psi'(s)\le t\}dt$.
Let $(Y,\mu)$ be a probability space. The Orlicz class is the space of measurable functions $f:Y\to\R$ or $\C$ such that
$$\varrho_{\psi}:=\int_Y\psi(|f(x)|)d\mu(x)<\infty.$$
We call $\varrho_{\psi}$ a modular function associated with $\psi$.
The Orlicz norm of a function $f$ is given by
\begin{equation}\label{onorm}
    \|f\|_{L^\psi(Y)}:=\sup \left\{\left|\int_Yfgd\mu\right|:g\in L_{\psi^*}(Y),\varrho_{\psi^*}(g)\le 1\right\}
\end{equation}
This norm is equivalent to the Luxemburg norm, which—by abuse of notation—we also denote by $\| \cdot \|_{L^\psi(Y)}$:
\begin{equation}\label{olnorm}
    \| f \|_{L^{\psi}(Y)} = \inf \left\{ \lambda > 0 \ : \ \int_{Y} \psi\left(\frac{|f(x)|}{\lambda}\right) d\mu(x) \leq 1 \right\}.
\end{equation}
The \emph{Orlicz-space} $L^\psi$ with assigned Young-function $\psi$ consists of the set of measurable functions $f$ such that $\|f\|_{L^\psi}$ is finite.
In the following we utilize certain properties of the Orlicz-norm and Young-functions, which can be found in chapter $2$ of \cite{SFMA}. For $1\le p<\infty$, if $\psi(x)=x^p$, then $L^\psi=L^p$.
Since we only consider finite-measure spaces, we have for $\psi\ll\varphi$ if and only if $ L^{\varphi}\subset L^{\psi}$ and
\begin{equation}\label{embedding}
    \|\cdot\|_{L^{\psi}}\ll\|\cdot\|_{L^{\varphi}}.
\end{equation}

For  $p\in(1,\infty),$ $q\in[1,\infty],$ the \emph{Lorentz-space} \( L^{p,q}(X) \) consists of all measurable functions \( f \) for which the norm

\[
\|f\|_{L^{p,q}} =
\begin{cases}
\left( \int_0^\infty \left( t^{1/p} f^*(t) \right)^q \frac{dt}{t} \right)^{1/q}, & \text{if } q < \infty, \\
\sup_{t > 0} \, t^{1/p} f^*(t), & \text{if } q = \infty,
\end{cases}
\]

is finite. Here, \( f^* \) denotes the \emph{non-increasing rearrangement} of \( f \), defined by

\[
f^*(t) = \inf \left\{ \lambda > 0 : d_f(\lambda) \leq t \right\},
\]
where $d_f(\lambda)=\mu(\{ x \in X : |f(x)| > \lambda \})$

We require the following properties:
We have \( L^{p,p} = L^p \), while for \( p < r \), the inclusion \( L^p \subsetneq L^{p,r} \subsetneq L^{p,\infty} \) holds, and
\begin{equation}\label{incl}
    \|\cdot\|_{L^{p,r}}\ll\|\cdot\|_{L^{p}}.
\end{equation}
Furthermore, assume for \( 1 \le p, p_1, p_2 < \infty \), \( 1 \leq q, q_1, q_2 \leq \infty \), that

\[
 \frac{1}{p}=\frac{1}{p_1} + \frac{1}{p_2} \quad \text{and}\quad \frac{1}{q}=\frac{1}{q_1} + \frac{1}{q_2}.
\]
Then, for \( f \in L^{p_1,q_2} \) and \( g \in L^{q_1,q_2} \) we have
\begin{equation}\label{Hol}
    \|fg\|_{L^{p,q}} \ll_{p_1,p_2,q_1,q_2} \, \|f\|_{L^{p_1,q_1}} \, \|g\|_{L^{p_2,q_2}}.
\end{equation}

For more details, see \cite{SRC,Mal}.

\section{Main Results}
In this section, we provide the precise statements of the main results of this paper. 
We mention that Rademacher-functions $r_n(t):=\textnormal{sign}(\sin(2\pi 2^n t))$ are viewed as random variables in the Probability space $([0,1],\lambda)$ with $\lambda$ being the Lebesgue measure.

\begin{lem}[Khinchin's inequality in Orlicz spaces]\label{khi}
    Let $r_n$ be the $n$-th Rademacher function. Let $\psi$ be a Young-function with $(\cdot)\ll_\psi\psi\ll_\psi e^{(\cdot)}$ and let $ x_1,\ldots,x_N\in \mathbb{C}$. Then
\begin{equation}\label{Khinch}
    \left\|\sum_{n=1}^N x_nr_n \right\|_{L^{\psi}}  \simeq_\psi  \left\|\sum_{n=1}^N x_nr_n \right\|_{L^{2}} = \left(\sum_{n=1}^N |x_n|^2\right)^{1/2}.
\end{equation}
\end{lem}

\begin{thm}[Marcinkiewicz–Zygmund inequality in Orlicz spaces]\label{mz}
Let $\{X_n\}_{n=1}^N$ be independent random variables with $\E\left[X_n\right]=0$ on a probability space $(Y,\mu)$. Then, for every Orlicz function $\psi$, with $(\cdot)\ll_\psi\psi(2\cdot)\ll_\psi \psi$, we have
\begin{equation}
        \left\|\sum_{n=1}^NX_n\right\|_{L^\psi}\simeq_\psi\left\|\left(\sum_{n=1}^N\left|X_n\right|^2\right)^\frac{1}{2}\right\|_{L^\psi}
\end{equation}
\end{thm}

\begin{rem}
    We will see in the proof, that for symmetric random variables the assumptions $(\cdot)\ll_\psi\psi\ll_\psi e^{(\cdot)}$ of Khinchin's inequality \eqref{khi} are sufficient.
\end{rem}


\begin{rem}
    Let $p\in[1,\infty)$. Then, for $\psi(x)=x^p$ the above results imply Khinchin's inequality, see \cite{KH} and the Marcinkiewicz–Zygmund inequality, see \cite{MZ1,MZ2}.
\end{rem}

An alternative generalization of classical \( L^p \)-spaces is given by the \emph{Lorentz spaces} \( L^{p,q} \) for $p\in(1,\infty),$ $q\in[1,\infty],$ which provide a finer scale of function spaces that interpolate between different \( L^p \) norms.


\begin{lem}[Khinchin's inequality in Lorentz spaces]\label{khi2}
    Let $r_n(t):=\textnormal{sign}(\sin(2\pi 2^n t))$, be the $n-$th Rademacher function. Let $p\in(1,\infty),$ $q\in[1,\infty],$ and let $ x_1,\ldots,x_N\in \mathbb{C}$. Then
\begin{equation}\label{Khinch2}
    \left\|\sum_{n=1}^N x_nr_n \right\|_{L^{p,q}}  \simeq_{p,q}  \left\|\sum_{n=1}^N x_nr_n \right\|_{L^{2}} = \left(\sum_{n=1}^N |x_n|^2\right)^{1/2}.
\end{equation}
\end{lem}

\begin{thm}[Marcinkiewicz–Zygmund inequality in Lorentz spaces]\label{mzls}
Let $\{X_n\}_{n=1}^N$ be an independent random variables with $\E\left[X_n\right]=0$. Then, for $p\in(1,\infty),$ $q\in[1,\infty],$ we have
\begin{equation}        \left\|\sum_{n=1}^NX_n\right\|_{L^{p,q}}\simeq_{p,q}\left\|\left(\sum_{n=1}^N\left|X_n\right|^2\right)^\frac{1}{2}\right\|_{L^{p,q}}.
\end{equation}
\end{thm}

\begin{rem}
    Proving the lower bound of Khinchin's inequality presents a difficulty in the case \( p = 1 \), as Hölder's inequality cannot be applied.
    An additional advantage of assuming \( p \neq 1 \) is that we remain within the framework of Banach spaces, avoiding potential complications that arise in quasi-Banach spaces.
\end{rem}

In the following we prove an extension of the well-known Kadec-Pe{\l}czy\'nski theorem, where we use the terminology of the first and last authors work \cite{BT}.

\begin{thm}[Generalized Kadec-Pe{\l}czy\'nski theorem]\label{kp}
    Let $\psi$ be a Young-function with $(\cdot)\ll_\psi\psi\ll_\psi (\cdot)^2$ and let $(X_n)_{n\in\N}$ be a determining sequence of random variables, such that $\|X_n\|_{L^\psi}=1$ for all $n\in\N$, $\{\psi\left(|X_n|\right), n\ge 1\}$ is uniformly integrable and $X_n\to 0$ weakly in $L^\psi$.
    Let $\mu$ be a limit random measure of $(X_n)_{n\in\N}$.

    Then there exists a subsequence $\left(X_{n_k}\right)$ equivalent to the unit vector basis of $l^2$ if and only if
    \begin{equation}
        \int_{\R}x^2d\mu(x)\in L^{\sqrt{\psi}}.
    \end{equation}

    \end{thm}

\begin{proof}[Proof of Theorem \ref{kp}]
The proof runs along the same lines as the proof of Theorem~1.4 in \cite{BST}, to appear in \emph{Studia Mathematica}.
\end{proof}

\section{Proof of Khinchin's Inequality \ref{Khinch} and \ref{khi2}}
We adapt the proof of Khinchin's Inequality of Muscalu and Schlag \cite[Lemma 5.5]{CS}.

\begin{proof}[Proof of Lemma \ref{khi}]
 By definition every Young-function is convex. A generalized version of Young's inequality, see \cite{HLP}, together with \eqref{onorm}
 impliy for any Young function $\psi$ and $f\in L^{\psi}$
 \begin{equation}\label{polyaineq}
     \|f\|_{L^\psi}\le\varrho_{\psi}(f)+1.
 \end{equation}
    Assume $\sum_{n=1}^Nx_n^2=1$ and $x_n\in\R$. Since $\{r_n\mid n\in\{1,\dots,N\}\}$ is a set of independent random variables, so is $\{e^{x_nr_n}\mid n\in\{1,\dots,N\}\}$. This implies that

    \begin{equation}
        \int_0^1e^{\pm\sum_{n=1}^Nx_nr_n(t)}dt=\prod_{n=1}^N\int_0^1e^{\pm x_nr_n(t)}dt=\prod_{n=1}^N\cosh(x_n)\le\prod_{n=1}^Ne^{x_n^2}=e,
    \end{equation}
hence,
\begin{equation}\label{2e}
        \int_0^1e^{\left|\sum_{n=1}^Nx_nr_n(t)\right|}dt\le 2e,
    \end{equation}
Utilizing equations \eqref{polyaineq}, \eqref{2e} and the assumption $\psi\ll_\psi e^{(\cdot)}$ yields
\begin{equation}
    \left\|\sum_{n=1}^N x_nr_n \right\|_{L^{\psi}}  \le \varrho_\psi\left(\sum_{n=1}^N x_nr_n \right) +1\ll_{\psi}\int_0^1e^{\left|\sum_{n=1}^Nx_nr_n(t)\right|}dt+1\ll_\psi 1
\end{equation}
Now let $x_n\in\R$ and define $\left(\sum_{n=1}^Nx_n^2\right)^\frac{1}{2}=\Upsilon>0$; $\Upsilon=0$ is not interesting. Then
\begin{equation}
    \left\|\sum_{n=1}^N \frac{x_nr_n }{\Upsilon}\right\|_{L^{\psi}}\ll 1,
\end{equation}
hence
\begin{equation}
    \left\|\sum_{n=1}^N x_nr_n \right\|_{L^{\psi}}\ll \left(\sum_{n=1}^Nx_n^2\right)^\frac{1}{2}.
\end{equation}
For $x_n\in\C$, the triangle inequality and the fact that $\Re(x_n)^2, \Im(x_n)^2\le |x_n|^2$ concludes the upper bound:
\begin{equation}\label{upper}
    \left\|\sum_{n=1}^N x_nr_n \right\|_{L^{\psi}}\le \left\|\sum_{n=1}^N \Re(x_n)r_n \right\|_{L^{\psi}}+\left\|\sum_{n=1}^N \Im(x_n)r_n \right\|_{L^{\psi}}\ll \left(\sum_{n=1}^N|x_n|^2\right)^\frac{1}{2}
\end{equation}
Let $S_N=\sum_{n=1}^Nx_nr_n$. The lower bound follows from applying Hölder's inequality and using the upper bound \eqref{upper}:
 \begin{equation}
     \|S_N\|_{L^2}\le \||S_N|^{\frac{1}{3}}\|_{L^3} \||S_N|^{\frac{2}{3}}\|_{L^6} \ll  \||S_N|\|_{L^1}^{\frac{1}{3}} \||S_N|\|_{L^2}^{\frac{2}{3}},
 \end{equation}
hence
\begin{equation}\label{conclusion}
    \|S_N\|_{L^2}\ll\|S_N\|_{L^1}\le \|S_N\|_{L^\psi}.
\end{equation}
\end{proof}

\begin{proof}[Proof of Lemma \ref{khi2}]
    By equation \eqref{incl} we have for $p\le q$
    \begin{equation}
        \left\|S_N \right\|_{L^{p,q}}  \ll  \left\|S_N \right\|_{L^{p}},
    \end{equation}
and for $q\le p$ the Hölder inequality \eqref{Hol} implies
\begin{equation}
        \left\|S_N \right\|_{L^{p,q}}  \ll_{p_1,q_1,p_2,q_2}  \left\|S_N \right\|_{L^{2p,2p}}=\left\|S_N \right\|_{L^{2p}},
    \end{equation}
where $p_1=p_2=2p$, and $q_1=(2p-q)/(2pq)$, $q_2=2p$. Thus, the implicit constant is still only dependent on $p$ and $q$. The upper bound follows now from the classical Khinchin inequality, where we gain constant factors only dependent on $p$.

For the lower bound it is sufficient to adapt the second inequality in equation \eqref{conclusion}, which can directly be done by the Hölder inequality. For $1=1/p_1+1/p_2$ and $1=1/q_1+1/q_2$, we choose $p_1=p/(p-1)$, $p_2=p$ and $q_1=q/(q-1)$, $q_2=q$. If we consider $q=1$, then we let $q_1=\infty$ and vice versa. Thus:
\begin{equation}
    \|S_N\|_{L^{1}}=\|S_N\|_{L^{1,1}}\ll_{p,q}\|S_N\|_{L^{p,q}}.
\end{equation}

\end{proof}

\section{Proof of the Marcinkiewicz-Zygmund inequality \eqref{mz}}
\begin{proof}[Proof of Theorem \ref{mz}]
    The proof follows closely the proof of the well-known Marcinkiewicz–Zygmund inequality, see e.g. \cite[Proposition 5.15]{CS} or \cite[Theorem 10.3.2]{CT}.

    \textbf{Step 1:} It is easy to check, that $\sum_{n=1}^NX_n\in L^\psi$ iff $X_i\in L^\psi$ for $i=1,\dots,N$, iff $\left(\sum_{n=1}^NX_n^2\right)^{1/2}\in L^\psi$, whence the latter may be supposed.

    \textbf{Step 2:} Let first $\{X_i\}_{n=1}^N$ be symmetric, i.e. $X_i=-X_i$, then 
    \begin{equation}     \left\|\sum_{n=1}^Nr_iX_i\right\|_{L^{\psi}}=\left\|\sum_{n=1}^NX_i\right\|_{L^{\psi}},
    \end{equation}
    where $r_i$ are Rademacher functions viewed as being independent of $\{X_i\}_{j=1}^N$. Khinchin's inequality \eqref{khi} implies the claim.

    \textbf{Step 3:} Let \(\tilde{X}_n := X_n - \acute{X}_n\) be the symmetrization of \(X_n\) for all \(1 \leq n \leq N\), where \(\{ \acute{X}_n \}_{n=1}^N\) are independent of and identically distributed with \(\{ X_n \}_{n=1}^N\). Then
    
    \begin{align}
    \begin{split}
        \left\|\sum_{i=n}^Nr_nX_n\right\|_{L^\psi}&\le\left\|\sum_{i=n}^Nr_n\tilde{X}_n\right\|_{L^\psi}\le \left\|2\max\left\{\sum_{i=n}^Nr_nX_n,\sum_{i=n}^Nr_n\acute{X}_n\right\}\right\|_{L^\psi}\\  &\ll_{\psi}\left\|\sum_{i=n}^Nr_nX_n\right\|_{L^\psi}+\left\|\sum_{i=n}^Nr_n\acute{X}_n\right\|_{L^\psi}=2\left\|\sum_{i=n}^Nr_nX_n\right\|_{L^\psi},
        \end{split}
      \end{align}
     where the first inequality can be shown identically to the original proof (see the references above), while the third follows from our assumption \(\psi(2\cdot) \ll_{\psi} \psi\). Applying Khinchin's inequality concludes the proof.
 \end{proof}

\begin{proof}[Proof of Theorem \ref{mzls}]
    Steps 1 and 2, as well as the first inequality in Step 3, are identical to the proof above. The remaining inequality is derived from the following argument:
     \begin{align}
    \begin{split}
        \left\|\sum_{i=n}^Nr_n\tilde{X}_n\right\|_{L^{p,q}}^q&\le \int_0^\infty t^\frac{1}{p}\left[\left(\sum_{i=n}^Nr_nX_n\right)^*+\left(\sum_{i=n}^Nr_n\acute{X}_n\right)^*\right]^q\\  &\le 2^{q-1}\left(\left\|\sum_{i=n}^Nr_nX_n\right\|_{L^{p,q}}^q+\left\|\sum_{i=n}^Nr_n\acute{X}_n\right\|_{L^{p,q}}^q\right)\simeq_q\left\|\sum_{i=n}^Nr_nX_n\right\|_{L^{p,q}}^q
        \end{split}
      \end{align}
      The first inequality follows from $d_{f+g}\le d_f+d_g$, and the second from Young’s inequality. Khinchin’s inequality \eqref{khi2} completes the proof. We remark that the ideas used in this proof are the same as those used in proving the triangle inequality for the Lorentz norm.
\end{proof}

\textbf{Acknowledgments}
ES expresses gratitude to Christoph Aistleitner and Andrei Shubin for helpful discussions.

\bibliographystyle{siam}






\end{document}